\documentclass[11pt]{amsart}
\usepackage{amsmath}
\usepackage{amsthm}
\usepackage{amssymb}
\usepackage{amsfonts,mathrsfs}
\usepackage{stmaryrd}
\usepackage{amsxtra}  
\usepackage{epsfig}
\usepackage{verbatim}

\theoremstyle{plain}
\newtheorem{theorem}[equation]{Theorem}

\newtheorem{lemma}[equation]{Lemma}
\newtheorem{corollary}[equation]{Corollary}

\theoremstyle{remark}

\numberwithin{equation}{section}

\newcommand{\sjump}{\hskip .2 cm}

\newcommand{\dbar}{\bar \partial}

\begin{document}

\title[Holomorphic $L^{1}$-functions]{Integrating holomorphic $L^{1}$-functions}
\author{A.-K. Herbig}
\thanks{Research supported by the Austrian Science Fund FWF grant  V187N13.}
\address{Department of Mathematics, \newline University of Vienna, Vienna, Austria}
\email{anne-katrin.herbig@univie.ac.at}
\begin{abstract}
  Let $\Omega\Subset\mathbb{C}^{n}$ be a domain with smooth boundary, $k\in\mathbb{N}$. It is shown that the integral of a holomorphic function in $L^{1}(\Omega)$ may be represented as the integral of this function  against a smooth function vanishing to order $k-1$ on $b\Omega$.  An application for a smoothing property of the Bergman projection for conjugate holomorphic functions is given.
\end{abstract}
\maketitle

\section{Introduction}
Let $\Omega\Subset\mathbb{C}^{n}$, $n\geq 1$, be a bounded domain with smooth boundary. Write $\mathcal{O}(\Omega)$ for the space of holomorphic functions on $\Omega$, and denote  by $\mathcal{C}^{\infty}(\overline{\Omega})$ the space of functions which are smooth up to the boundary, $b\Omega$, of $\Omega$. The Bergman projection
\begin{align*}
  B: L^{2}(\Omega)\longrightarrow L^{2}(\Omega)\cap\mathcal{O}(\Omega)
\end{align*} 
is the orthogonal projection of $L^{2}(\Omega)$, the space of square-integrable functions on $\Omega$, onto its closed subspace of holomorphic functions.

In \cite{Bell}, Section 6,  Bell constructed for each $k\in\mathbb{N}$ a differential operator $\Phi^{k}$ of order $k$, with coefficients in $\mathcal{C}^{\infty}(\overline{\Omega})$, satisfying
\begin{itemize}
  \item[(i)] $B(\Phi^{k}f)=Bf$ for all $f\in\mathcal{C}^{\infty}(\overline{\Omega})$,
  \item[(ii)]  $\Phi^{k}f$ vanishes to order $k-1$ on $b\Omega$ whenever $f\in\mathcal{C}^{\infty}(\overline{\Omega})$.
\end{itemize}
Since $\Omega$ is bounded, $\mathcal{C}^{\infty}(\overline{\Omega})\subset L^{2}(\Omega)\subset L^{1}(\Omega)$ holds. Hence it may be deduced that for $\eta\in\mathcal{C}^{\infty}(\overline{\Omega})\cap\mathcal{O}(\Omega)$
\begin{align*}
  \int_{\Omega}\eta \,dV= \int_{\Omega}B\eta \,dV\stackrel{\text{(i)}}{=} \int_{\Omega}B(\Phi^{k}\eta) \,dV= \int_{\Omega}\bigl(\Phi^{k}\eta\bigr)\cdot \overline{B(1)} \,dV
\end{align*}
holds. Here, in the  last step, the self-adjointness of $B$ (with respect to the Hermitian $L^{2}$-inner product) was used. Since $B(1)=1$ it follows that
\begin{align}\label{E:Bellnew}
  \int_{\Omega}\eta\,dV=\int_{\Omega}\Phi^{k}\eta\,dV\sjump\qquad\forall\sjump 
  \eta\in\mathcal{C}^{\infty}(\overline{\Omega})\cap\mathcal{O}(\Omega).
\end{align}

The purpose of this note is to extend the integral identity \eqref{E:Bellnew} to a wider class of functions while replacing the differential operator $\Phi^{k}$  on its right hand side by multiplication with a smooth function which vanishes to order $k-1$ on $b\Omega$. 

\begin{theorem}\label{T:main}
  Let $\Omega\Subset\mathbb{C}^{n}$, $n\geq 1$, be a smoothly bounded domain.  
  Let $\delta\in\mathcal{C}^{\infty}(\overline{\Omega})$ be a function which equals the distance-to-the-boundary function for 
  $\Omega$  near $b\Omega$.
  Let  $k\in\mathbb{N}$ and $g\in\mathcal{C}^{\infty}(\overline{\Omega})$ be given. Then there exists  a function $\omega_{k,g}\in\mathcal{C}^{\infty}(\overline{\Omega})$ such that
  \begin{align}\label{E:main}
    \int_{\Omega} \eta\cdot g\,dV=\int_{\Omega}\delta^{k}\cdot\omega_{k,g}\cdot\eta \,dV\qquad\sjump\forall\sjump\eta\in 
    L^{1}(\Omega)\cap \mathcal{O}(\Omega).
  \end{align}
\end{theorem}

In the special case when $\eta\in L^{2}(\Omega)\cap\mathcal{O}(\Omega)$ and $g\equiv 1$, identity \eqref{E:main} may be derived directly from \eqref{E:Bellnew} by an integration by parts argument. E.g., when $k=1$, 
one first notices that Bell's operator $\Phi^{1}$ is of the form $r\widetilde{\Phi}^{1}$ where $r\in\mathcal{C}^{\infty}(\overline{\Omega})$ is zero on $b\Omega$ and $\widetilde{\Phi}^{1}$ is a first order differential operator with coefficients in
 $\mathcal{C}^{\infty}(\overline{\Omega})$. It then follows from the ellipticity of the $\dbar$-operator on functions that $\widetilde{\Phi}^{1}\eta$ may be written as a tangential derivative of $\eta$ (plus a lower order term) so that integrating by parts yields a vanishing boundary term (see Lemma 6.1 in  \cite{Boas} for details for this kind of reasoning). 
 This argument lets one rewrite the right hand side  of \eqref{E:Bellnew} in the shape of the right hand side of \eqref{E:main}.

\medskip

Theorem \ref{T:main} allows us to reprove, and slightly improve,  a smoothing property of the Bergman projection  for conjugate holomorphic functions in $L^{2}(\Omega)$ previously proved as Theorem 1.10 in \cite{HMS}. To state this properly, let us fix some notation first. For given $\ell\in\mathbb{Z}$ denote by $H^{\ell}(\Omega)$ the $L^{2}$-Sobolev  space of order $\ell$, write $\|.\|_{\ell}$ for its norm, see Subsection \ref{SS:Settingagain} for more details. Furthermore, set $A^{0}(\Omega)=L^{2}(\Omega)\cap\mathcal{O}(\Omega)$, and define $\overline{A^{0}(\Omega)}$ to be the space of  functions  whose complex conjugates belong to $A^{0}(\Omega)$.
\begin{corollary}\label{C:smoothing}
  Let $\Omega\Subset\mathbb{C}^{n}$, $n\geq 1$, be a smoothly bounded domain. Suppose that for a given pair $(k_{1},k_{2})\in
  \mathbb{N}_{0}^{2}$ there exists a constant $C>0$ such that
  \begin{align}\label{E:regassumption}
    \left\|Bf\right\|_{k_{2}}\leq C \|f\|_{k_{1}}\sjump\qquad\sjump\forall \sjump f\in H^{k_{2}}(\Omega).
  \end{align}
  Let $k\in\mathbb{N}$ and $g\in\mathcal{C}^{\infty}(\overline{\Omega})$ be given. Then there exists a constant $\widetilde{C}>0$ such that
  \begin{align}\label{E:smoothing}
    \left\|B(\mu g)\right\|_{k_{2}}\leq \widetilde{C}\|\mu\|_{-k}\sjump\qquad\forall\sjump\mu\in\overline{A^{0}(\Omega)}.
  \end{align}
\end{corollary}

Corollary \ref{C:smoothing} has been shown to hold in the case $k=0$ in \cite{HMS}, Theorem 1.10. This case indicates that the Bergman projection maps the products of the form $\mu g$ to  particularly ``nice'' holomorphic $L^{2}$-functions. That \eqref{E:smoothing} holds for any $k\in\mathbb{N}$ seems to hint at $B$ acting on such a product  in a ``simple'' manner. To be less vague the introduction of  some more notation is convenient. Denote by $\overline{A^{-k}(\Omega)}$ the closure of $\overline{A^{0}(\Omega)}$ with respect to $\|.\|_{-k}$. Additionally, write $\overline{A_{g}^{-k}(\Omega)}$ for the space consisting of functions which are  products of $g$ and functions in  $\overline{A^{-k}(\Omega)}$.

\begin{corollary}\label{C:extending}
 Suppose the hypotheses of Corollary \ref{C:smoothing} hold. Then the Bergman projection extends to an $H^{k_{2}}(\Omega)$-bounded operator on $\overline{A_{g}^{-k}(\Omega)}$.
\end{corollary}

The article is structured as follows. In Section \ref{S:proof} the proof of Theorem~\ref{T:main} is given.  The proofs of Corollaries \ref{C:smoothing} and \ref{C:extending} are provided in Section \ref{S:Corollaries}.

\subsection*{Acknowledgement} I would like to thank Jeff McNeal for pointing out to me that Corollary \ref{C:extending} holds and for his advice on the exposition of this article.

\medskip

\section{Proof of Theorem \ref{T:main}}\label{S:proof}
\subsection{Setting}\label{SS:Setting}
Before presenting the proof of Theorem \ref{T:main} let us give the basic definitions and notions used in this setting.

Let $\Omega\Subset\mathbb{C}^{n}$ be a domain with $\mathcal{C}^{\infty}$-boundary. Denote by $(z_{1},\ldots,z_{n})$ the standard coordinates of $\mathbb{C}^{n}$, write $x_{2j-1}=\text{Re}(z_{j})$ and 
 $x_{2j}=\text{Im}(z_{j})$ for $j\in\{1,\ldots,n\}$. Let $\|.\|$ be the Euclidean norm on $\mathbb{C}^{n}$.
The distance-to-the-boundary function, $d_{b\Omega}(z)$, for $\Omega$ is given by $\inf\{\|z-w\|:w\in b\Omega\}$
and satisfies the properties
\begin{itemize}
  \item[(a)] there exists a neighborhood $U$ of $b\Omega$ such that $d_{b\Omega}$ is smooth on $\overline{\Omega}\cap U$,
  \item[(b)] $\|\nabla d_{b\Omega}(z)\|=1$ for all $z\in\overline{\Omega}\cap U$.
\end{itemize}
 For proofs of  the facts (a) and (b), see Lemma 1, pg. 382, in \cite{GT}. 
 It follows from (a) that
  \begin{align}\label{E:normalvfield}
   N:=\sum_{j=1}^{2n}\frac{\partial d_{b\Omega}}{\partial x_{j}}\frac{\partial}{\partial x_{j}}
  \end{align}
 is a smooth vector field  on $\overline{\Omega}\cap U$. 
 Moreover, (b) implies that $N(d_{b\Omega})=1$ on $\overline{\Omega}\cap U$.
 
 As usual, for an open set  $U\subset\mathbb{C}^{n}$, let $\mathcal{C}^{\infty}_{c}(U)$ be the space of functions in $\mathcal{C}^{\infty}(\overline{U})$ which are compactly supported in $U$. Also,  $L^{1}(\Omega)$ is the space of measurable functions
 $f:\Omega\longrightarrow\mathbb{C}$ satisfying
 \begin{align*}
   \int_{\Omega}|f|\,dV<\infty,
 \end{align*}
 where $dV$ is the Euclidean volume form.
   
  \subsection{Base case}\label{SS:basecase}
   The proof of Theorem \ref{T:main} will be done by induction on $k$. The base case, $k=1$, is covered by the following lemma.
  \begin{lemma}\label{L:basecase}
    Let $\Omega\Subset\mathbb{C}^{n}$, $n\geq 1$, be a smoothly bounded domain. 
    Let $\delta\in\mathcal{C}^{\infty}(\overline{\Omega})$ be a function which equals $d_{b\Omega}$ near $b\Omega$.
    For any given $\gamma \in\mathcal{C}^{\infty}(\overline{\Omega})$ there exists a function
    $\omega_{1,\gamma}\in\mathcal{C}^{\infty}(\overline{\Omega})$ such that
    \begin{align}\label{E:basecase}
      \int_{\Omega}\eta\cdot\gamma\,dV=\int_{\Omega}\delta\cdot\omega_{1,\gamma}\cdot\eta\,dV\sjump\qquad\forall\sjump\eta\in L^{1}(\Omega)
      \cap\mathcal{O}(\Omega).
    \end{align}  
\end{lemma}

\begin{proof}[Proof of Lemma \ref{L:basecase}]
  After possibly shrinking the neighborhood $U$ of $b\Omega$ described in subsection \ref{SS:Setting}, it may be assumed that $\delta$ equals $d_{b\Omega}$ on $\overline{\Omega}\cap U$. 
  Now choose a function $\zeta\in\mathcal{C}_{c}^{\infty}(U)$ such that  $\zeta\equiv 1$ in some 
  neighborhood $U'\Subset U$ of $b\Omega$. Then
  \begin{align*}
    \int_{\Omega}\eta \gamma \,dV=\int_{\Omega}\zeta \eta \gamma\,dV+\underbrace{\int_{\Omega}(1-\zeta)\eta \gamma\,dV}_{=:\text{I}_{1}}.
  \end{align*}
  Note that the term
  \begin{align*}
    \text{I}_{1}=\int_{\Omega}\delta\left( \frac{1-\zeta}{\delta}\cdot \gamma\right)\eta\;dV
  \end{align*}
 is of the shape as the claimed  right hand side of \eqref{E:basecase} since 
 $1-\zeta\equiv 0$ in the neighborhood $U'$ of $b\Omega$. The fact that $N(\delta)=N(d_{b\Omega})=1$ on $\Omega\cap U$ gives
  \begin{align}\label{E:Stokes1}
    \int_{\Omega}\eta \gamma \,dV=\int_{\Omega}N(\delta)\zeta\eta \gamma\,dV+\text{I}_{1}.
  \end{align}
  The intermediate goal is to peel $N$ off $\delta$ by an integration by parts argument so that $\delta$  becomes a multiplicative factor.  For that, notice first that the product rule yields
\begin{align*}  
  N(\delta)\zeta\eta \gamma&=N(\delta\zeta\eta \gamma)-\delta N(\zeta\eta \gamma)\\
  &\stackrel{\eqref{E:normalvfield}}{=}\sum_{j=1}^{2n}\frac{\partial}{\partial x_{j}}\left(\delta_{x_{j}}\delta\zeta\eta \gamma\right)
  -\delta(\triangle\delta)\zeta\eta \gamma-\delta N(\zeta\eta \gamma).
\end{align*}  
Integrating over $\Omega$ gives
\begin{align}\label{E:Stokes2}
  \int_{\Omega} N(\delta)\zeta\eta \gamma\,dV
  =-\int_{\Omega}\delta N(\zeta\eta \gamma)\,dV&-\int_{\Omega}\delta(\triangle\delta)\zeta\eta \gamma\,dV\\
  &+\int_{\Omega}d\Bigl(\delta\zeta\eta \gamma\sum_{j=1}^{2n}(-1)^{j+1}\delta_{x_{j}}\,d\widehat{x_{j}} \Bigr),\notag
\end{align}
where $d\widehat{x_{j}}=dx_{1}\wedge\ldots \wedge dx_{j-1}\wedge dx_{j+1}\wedge\ldots\wedge dx_{2n}$.
If $\eta$ was smooth up to $b\Omega$, then an application of Stokes' Theorem and the fact that $\delta=0$ on $b\Omega$ would imply that the last term on the right hand side of \eqref{E:Stokes1} is zero. However, $\eta\in L^{1}(\Omega)\cap\mathcal{O}(\Omega)$, and hence some extra care has to be taken.  Let $\Omega_{\epsilon}=\{z\in\Omega:d_{b\Omega}(z)>\epsilon \}$ for $\epsilon>0$  small. Since both $\eta$ and $\gamma$ are smooth in $\Omega$, and therefore on $\overline{\Omega}_{\epsilon}$, Stokes' theorem may be used as follows:
\begin{align*}
 \int_{\Omega_{\epsilon}}d\Bigl(\delta\zeta\eta \gamma\sum_{j=1}^{2n}(-1)^{j+1}\delta_{x_{j}}\,d\widehat{x_{j}} \Bigr)
  &=\epsilon\int_{b\Omega_{\epsilon}}\zeta\eta \gamma\sum_{j=1}^{2n}(-1)^{j+1}\delta_{x_{j}}\,d\widehat{x_{j}}\\
  &=\epsilon  \int_{\Omega_{\epsilon}}d\Bigl(\zeta\eta \gamma\sum_{j=1}^{2n}(-1)^{j+1}\delta_{x_{j}}\,d\widehat{x_{j}} \Bigr).
\end{align*}   
Therefore,
 \begin{align*}  
  \int_{\Omega_{\epsilon}}d\Bigl(\delta\zeta\eta \gamma\sum_{j=1}^{2n}(-1)^{j+1}\delta_{x_{j}}\,d\widehat{x_{j}} \Bigr)
   &=\epsilon\underbrace{\int_{\Omega_{\epsilon}}N(\zeta\gamma)\eta\,dV}_{=:\text{I}_{2}(\epsilon)}+\epsilon
   \underbrace{\int_{\Omega_{\epsilon}}\zeta\gamma N(\eta)\,dV}_{=:\text{I}_{3}(\epsilon)}.
\end{align*} 
Since $\eta\in L^{1}(\Omega)$ and both $\zeta$ and $\gamma$ are smooth up to $b\Omega$, it follows that $\text{I}_{2}(\epsilon)$ is uniformly bounded in $\epsilon$. Therefore 
$\epsilon \text{I}_{2}(\epsilon)$ approaches $0$ as $\epsilon$ goes to $0^{+}$. To understand the behavior of the term 
$\text{I}_{3}(\epsilon)$, the holomorphicity of  $\eta$ will be used to transform $N\eta$ into a tangential derivative of $\eta$. For that define
\begin{align*}
  T:=\sum_{j=1}^{n}\delta_{x_{2j}}\frac{\partial}{\partial x_{2j-1}}-\delta_{x_{2j-1}}\frac{\partial}{\partial x_{2j}}.
\end{align*}
Then $T$ is a smooth vector field on $\overline{\Omega}$. Moreover, $T$ is tangential to the level sets of $\delta$ 
as  $T(\delta)\equiv 0$ holds. Additionally, the holomorphicity of $\eta$ gives
\begin{align*}
  N\eta=\sum_{j=1}^{n}\left(\delta_{x_{2j-1}}\eta_{x_{2j-1}}+\delta_{x_{2j}}\eta_{x_{2j}}\right)
  =i\sum_{j=1}^{n}\left(-\delta_{x_{2j-1}}\eta_{x_{2j}}+\delta_{x_{2j}}\eta_{x_{2j-1}}\right)=iT\eta.
\end{align*}
Hence
\begin{align}\label{E:Stokes4}
  \text{I}_{3}(\epsilon)=i\int_{\Omega_{\epsilon}}\zeta\gamma T(\eta)\,dV
  =i\int_{\Omega_{\epsilon}}T(\zeta\gamma\eta)\,dV-i\int_{\Omega_{\epsilon}}T(\zeta\gamma)\eta\,dV.
\end{align}
The last term on the right hand side of \eqref{E:Stokes4} is uniformly bounded in $\epsilon$ since $\eta\in L^{1}(\Omega)$ and $\zeta, \gamma\in\mathcal{C}^{\infty}(\overline{\Omega})$.
Moreover, the second to last term on the right hand side actually is zero. This follows from yet another application of Stokes' theorem and the facts that $T$ is tangential to $b\Omega_{\epsilon}$, self-adjoint and $\eta, \gamma, \zeta\in\mathcal{C}^{\infty}(\Omega)$. Therefore, $\text{I}_{3}(\epsilon)$ goes to $0$ as $\epsilon\to 0^{+}$. This concludes the proof of
 the last term on the right hand side of \eqref{E:Stokes2} being  zero. 
 
 Combining the two  identities  \eqref{E:Stokes1} and \eqref{E:Stokes2} then gives
\begin{align*}
   \int_{\Omega}\eta \gamma \,dV&=-\int_{\Omega}\delta N(\zeta\eta \gamma)\,dV
   -\int_{\Omega}\delta(\triangle\delta)\zeta\eta \gamma\,dV+\text{I}_{1}\\
   &=-\int_{\Omega}\delta N(\eta)\zeta \gamma\,dV\underbrace{-\int_{\Omega}\delta\eta\left(N(\zeta \gamma)+(\triangle\delta)\zeta \gamma\right)\,dV}_{=:\text{I}_{4}}
   +\text{I}_{1}.
\end{align*}
Note that the term $\text{I}_{4}$ is of the shape as the claimed right hand side of \eqref{E:basecase}.  So the intermediate goal of peeling $N$ off $\delta$ has been achieved at the cost of introducing an $N$-derivative of $\eta$. However, 
the holomorphicity of $\eta$ lets us repeat the trick leading up to \eqref{E:Stokes4} to obtain
\begin{align}\label{E:Stokes5}
  \int_{\Omega}\eta \gamma \,dV&=-i\int_{\Omega}\delta T(\eta)\zeta \gamma\,dV+\text{I}_{1}+\text{I}_{4}\notag\\
  &=-i\int_{\Omega}T(\delta \eta\zeta \gamma)\,dV+\underbrace{i\int_{\Omega}\delta\eta T(\zeta \gamma)\,dV}_{=:\text{I}_{5}}+\text{I}_{1}+\text{I}_{4},
\end{align}
where $\text{I}_{5}$ is of the right shape with respect to the claimed \eqref{E:basecase}. Also, arguments analogously to the ones following \eqref{E:Stokes4} result in the first term on the right hand side of \eqref{E:Stokes5} being equal to zero.

 Hence it is shown that
\begin{align*}
  \int_{\Omega}\eta \gamma\,dV=\text{I}_{1}+\text{I}_{4}+\text{I}_{5},
\end{align*}
i.e., \eqref{E:basecase} holds with $$\omega_{1,\gamma}=\frac{1-\zeta}{\delta}\gamma-\left(N(\zeta \gamma)+(\triangle\delta)\zeta \gamma-iT(\zeta \gamma)\right).$$
 \end{proof}
 
 \subsection{Proof of Theorem \ref{T:main}} \label{SS:proofofmain}
 The proof is done by induction on $k$. The case $k=1$ is Lemma \ref{L:basecase} with $\gamma=g$.  
 Suppose now that for some $k\in\mathbb{N}$ there exists a function $\omega_{k,g}\in\mathcal{C}^{\infty}(\overline{\Omega})$ such that
 \begin{align}\label{E:indhypo}
   \int_{\Omega}\eta\cdot g\,dV=\int_{\Omega}\delta^{k}\cdot\omega_{k,g}\cdot\eta\,dV\sjump\qquad\forall\sjump\eta\in L^{1}(\Omega)\cap\mathcal{O}(\Omega)
 \end{align}
 holds.
  Let the neighborhood $U$ of $b\Omega$ and the cut-off function $\zeta$ be given as in the proof of Lemma \ref{L:basecase} and proceed analogously to the arguments leading up to \eqref{E:Stokes1}. That is,  
 \begin{align*}
   \int_{\Omega}\eta g\,dV\stackrel{\eqref{E:indhypo}}{=}\int_{\Omega}\delta^{k}\omega_{k,g}\eta\,dV
   =\int_{\Omega}\delta^{k}\zeta\eta\omega_{k,g}\,dV+\int_{\Omega}\delta^{k}(1-\zeta)\omega_{k,g}\eta\,dV.
\end{align*}   
Moreover,  since $N(\delta)\equiv 1$ on $\Omega\cap U$, it follows that
 \begin{align*}  
    \int_{\Omega}\eta g\,dV=\frac{1}{k+1}\int_{\Omega}N(\delta^{k+1})\zeta\omega_{k,g}\eta\,dV+
    \int_{\Omega}\delta^{k+1}\frac{1-\zeta}{\delta}\omega_{k,g}\eta\,dV.
 \end{align*}
 Now repeat the arguments starting from \eqref{E:Stokes1} with $\delta^{k+1}$ in place of $\delta$ and $\omega_{k,g}/(k+1)$ in place of $\gamma$ there to obtain
 \begin{align*}
   \int_{\Omega} \eta\cdot g\;dV=\int_{\Omega}\delta^{k+1}\cdot\omega_{k+1,g}\cdot\eta\,dV
 \end{align*}
 with
 \begin{align*}
   \omega_{k+1,g}=\frac{1-\zeta}{\delta}\omega_{k,g}-\frac{1}{k+1}\left(N(\zeta\omega_{k})+(\triangle\delta)\omega_{k,g}-iT(\zeta\omega_{k,g}) \right).
 \end{align*}
\qed
 
 \section{Proof of Corollaries \ref{C:smoothing} and \ref{C:extending}}\label{S:Corollaries}
 \subsection{Setting}\label{SS:Settingagain}
   Let $\Omega\Subset\mathbb{C}^{n}$ be a domain with $\mathcal{C}^{\infty}$-boundary. Define, as usual,
 $$L^{2}(\Omega)=\bigl\{f:\Omega\longrightarrow\mathbb{C}: f \text{\,measurable}, \int_{\Omega}|f|^{2}\,dV<\infty\bigr\}.$$  Write $(.,.)$ for the usual Hermitian $L^{2}$-inner product, i.e., $$(f,g)=\int_{\Omega}f\cdot \overline{g}\,dV$$ for $f,g\in L^{2}(\Omega)$.
 Denote by $\|.\|$ the induced norm. For a multi-index $\alpha=(\alpha_{1},\ldots,\alpha_{2n})\in\mathbb{N}_{0}^{2n}$ of length
 $|\alpha|=\sum_{j=1}^{2n}\alpha_{j}$ define the differential operator 
 \begin{align*}
   D^{\alpha}=\frac{\partial^{|\alpha|}}{\partial x_{1}^{\alpha_{1}}\ldots\partial x_{2n}^{\alpha_{2n}}}.
 \end{align*}
 The $L^{2}$-Sobolev space, $H^{k}(\Omega)$, of order $k\in\mathbb{N}$ is defined as
 $$H^{k}(\Omega)=\left\{f\in L^{2}(\Omega): D^{\alpha}f\in L^{2}(\Omega)\,\,\text{for all\,\,}\alpha\,\,\text{with\,\,}|\alpha|\leq k\right\};$$
 here $D^{\alpha}f$ for $f\in L^{2}(\Omega)$ is understood in the sense of distributions. $H^{k}(\Omega)$ is equipped with the inner product
 \begin{align*}
   \left(f,g\right)_{k}:=\sum_{|\alpha|\leq k}\left(D^{\alpha}f,D^{\alpha}g\right)\qquad\text{for}\sjump f,g\in H^{k}(\Omega),
 \end{align*}
 and is in fact a Hilbert space with the induced norm  $\|.\|_{k}$.
 Set $A^{k}(\Omega)$ to be the subspace $H^{k}(\Omega)\cap\mathcal{O}(\Omega)$, and denote by $\overline{A^{k}(\Omega)}$ the space of functions whose complex conjugates belong to $A^{k}(\Omega)$.
 
 Let $H_{0}^{k}(\Omega)$ be the closure of $\mathcal{C}^{\infty}_{c}(\Omega)$ with respect to $\|.\|_{k}$. Denote by $H^{-k}(\Omega)$ the dual space of $H^{k}_{0}(\Omega)$. The space $H^{-k}(\Omega)$ is then endowed with the operator norm, i.e., if $f\in H^{-k}(\Omega)$ then
 \begin{align*}
   \|f\|_{-k}=\sup\left\{\left|\left(f,g\right)\right|: g\in H_{0}^{k}(\Omega), \|g\|_{k}\leq 1 \right\}.
 \end{align*}
 For functions in  $A^{0}(\Omega)$ this norm reduces to an $L^{2}$-norm weighted with the appropriate power of the distance-to-the-boundary function. That is, there exists a constant $c>0$ such that
 \begin{align}\label{E:equivnorm}
    \frac{1}{c}\|h\|_{-k}\leq\|\delta^{k}h\|\leq c\|h\|_{-k}\sjump\qquad\forall\sjump h\in A^{0}(\Omega).
 \end{align}
 Clearly, \eqref{E:equivnorm} also holds for all $h\in\overline{A^{0}(\Omega)}$.
 Write $\overline{A^{-k}(\Omega)}$ for the closure of $\overline{A^{0}(\Omega)}$ with respect to $\|.\|_{-k}$.
 
 \medskip
 
 \subsection{Proof of Corollary \ref{C:smoothing}}\label{SS:smoothing}
   Suppose \eqref{E:regassumption} holds, i.e., for a given pair $(k_{1},k_{2})\in\mathbb{N}_{0}^{2}$ there exists a constant $C>0$ 
   such that
   \begin{align*}
     \left\|Bf\right\|_{k_{2}}\leq C\left\|f\right\|_{k_{1}}\qquad\sjump\forall\sjump f\in H^{k_{1}}(\Omega).
   \end{align*}
    Proposition 2.3 in \cite{HMS} states that under this regularity assumption on $B$ there exists a constant $c_{1}>0$ such that
   \begin{align}\label{E:duality}
     \|f\|_{k_{2}}\leq c_{1}\sup\left\{\left|\left(f,h\right) \right|: h\in A^{k_{2}}(\Omega), \|h\|_{-k_{1}} \right\}
   \end{align}
   for all $f\in A^{0}(\Omega)$. The case $k_{1}=0=k_{2}$ is not contained in \cite{HMS}, however it is obtained easily from 
   \begin{align*}
     \|f\|=\|Bf\|&=\sup\left\{\left|\left(Bf,g\right)\right|: g\in L^{2}(\Omega), \|g\|\leq 1 \right\}\\
     &=\sup\left\{\left|\left(f,Bg\right)\right|: g\in L^{2}(\Omega), \|g\|\leq 1 \right\}
   \end{align*}
   and the fact that $B$ is bounded in $L^{2}(\Omega)$.
   
   \medskip
   
   Let $k\in\mathbb{N}$ and $g\in\mathcal{C}^{\infty}(\overline{\Omega})$ be given. It then follows from \eqref{E:duality} that
   \begin{align*}
     \left\|B(\mu g) \right\|_{k_{2}}&\leq c_{1}\sup\left\{\left|\left(B(\mu g),h \right) \right| : h\in A^{k_{2}}(\Omega), \|h\|_{-k_{1}}\leq 1\right\}\\
     &= c_{1}\sup\left\{\left|\left(\mu g,h \right) \right| : h\in A^{k_{2}}(\Omega), \|h\|_{-k_{1}}\leq 1\right\}
   \end{align*}
   for all  $\mu\in \overline{A^{0}(\Omega)}$.
   Observe first that $\overline{\mu}h$ is holomorphic in $\Omega$, then notice that H\"older's inequality implies that 
   $\overline{\mu}h\in L^{1}(\Omega)$. Hence Theorem \ref{T:main} is applicable  here. In particular, for any  $k\in\mathbb{N}$ there exists a function $\omega_{k+k_{1},g}\in\mathcal{C}^{\infty}(\overline{\Omega})$ such that
   \begin{align*}
     \left(\mu g,h \right)=\left(\delta^{k}\cdot\omega_{k+k_{1},g}\cdot\mu ,\delta^{k_{1}}\cdot h \right).
   \end{align*}
   Applying (the right side of) \eqref{E:equivnorm} twice after using the Cauchy--Schwarz inequality then gives
   \begin{align*}
      \left| \left(\mu g,h \right)\right|\leq \bigl\|\delta^{k}\omega_{k+k_{1},g}\mu \bigr\|\cdot\bigl\|\delta^{k_{1}}h \bigr\|
      \leq c_{2}\|\mu\|_{-k}\cdot\|h\|_{-k_{1}}.
   \end{align*}
   Here $c_{2}>0$ is a constant depending on $\omega_{k+k_{1},g}$ and on the constant $c$ in \eqref{E:equivnorm}. Thus
   \begin{align*}
     \left\| B(\mu g)\right\|_{k_{2}}\leq c_{1}c_{2}\|\mu\|_{-k}.
   \end{align*}
\qed

\subsection{Proof of Corollary \ref{C:extending}}\label{SS:extending}
  Let $\mu\in\overline{A^{-k}(\Omega)}$. Then by definition there exists a sequence $\{\mu_{j}\}_{j}\in\overline{A^{0}(\Omega)}$ which converges to $\mu$ in $\|.\|_{-k}$. Moreover, it follows from \eqref{E:smoothing} that
\begin{align*}
  \left\|B(\mu_{j}g)-B(\mu_{\ell}g)\right\|_{k_{2}}=\left\|B\left((\mu_{j}-\mu_{\ell})g\right)\right\|_{k_{2}}
  \leq\widetilde{C}\|\mu_{j}-\mu_{\ell}\|_{-k}.
\end{align*}
Therefore $\{B(\mu_{j}g)\}_{j}$ is a Cauchy sequence with respect to $\|.\|_{k_{2}}$ and hence has a limit in $H^{k_{2}}(\Omega)$. Thus an extension, $\widetilde{B}$, of the Bergman projection $B$ may be defined by setting $\widetilde{B}(\mu g)$ to be the limit of 
$\{B(\mu_{j}g)\}_{j}$ in $H^{k_{2}}(\Omega)$.
It then follows that for any $\epsilon>0$ there exists a $j_{0}\in\mathbb{N}$ such that for all $j\geq j_{0}$
\begin{align*}
  \bigl\|\widetilde{B}(\mu g) \bigr\|_{k_{2}}\leq
  \bigl\|\widetilde{B}(\mu g)-B(\mu_{j}g) \bigr\|_{k_{2}}&+\bigl\|B(\mu_{j} g) \bigr\|_{k_{2}}\\
 & \leq\epsilon+\widetilde{C}\|\mu_{j}\|_{-k}\\
  &\leq \epsilon+\widetilde{C}\left(\|\mu-\mu_{j}\|_{-k}+\|\mu\|_{-k} \right)\\
  &\leq\epsilon(1+\widetilde{C})+\widetilde{C}\|\mu\|_{-k}.
\end{align*}
holds. Therefore, $\widetilde{B}$ is a bounded in $\|.\|_{k_{2}}$.
\qed

\end{document}